\documentclass{amsart}

\usepackage{graphicx}

\newtheorem{theorem}{Theorem}[section]
\newtheorem{lemma}[theorem]{Lemma}
\newtheorem{proposition}{Proposition}[section]
\theoremstyle{definition}
\newtheorem{definition}[theorem]{Definition}

\theoremstyle{remark}
\newtheorem{remark}[theorem]{Remark}

\numberwithin{equation}{section}

\begin{document}

\title{A new proof of some matrix inequalities}

\author{T. Agama}
\address{Department of Mathematics, African Institute for Mathematical science, Ghana
}
\email{theophilus@aims.edu.gh/emperordagama@yahoo.com}



\date{\today}



\begin{abstract}
In this paper we give alternate proofs of some well-known matrix inequalities. In particular, we show that under certain conditions the inequality holds \begin{align}\sum \limits_{\lambda_i\in \mathrm{Spec}(ab^{T})}\mathrm{min}\{\log |t-\lambda_i|\}_{[||a||,||b||]}&\leq \# \mathrm{Spec}(ab^T)\log\bigg(\frac{||b||+||a||}{2}\bigg)\nonumber \\&+\frac{1}{||b||-||a||}\sum \limits_{\lambda_i\in \mathrm{Spec}(ab^T)}\log \bigg(1-\frac{2\lambda_i}{||b||+||a||}\bigg).\nonumber
\end{align}Also under the same condition, the inequality also holds\begin{align}\int \limits_{||a||}^{||b||}\log|\mathrm{det}(ab^{T}-tI)|dt&\leq \# \mathrm{Spec}(ab^T)(||b||-||a||)\log\bigg(\frac{||b||+||a||}{2}\bigg)\nonumber \\&+\sum \limits_{\lambda_i\in \mathrm{Spec}(ab^T)}\log \bigg(1-\frac{2\lambda_i}{||b||+||a||}\bigg).\nonumber
\end{align}
\end{abstract}

\maketitle

\section{\textbf{Introduction and statement}}
There are useful inequalities in the literature that are quite hard to establish in a normal fashion, and it is often the case that a direct approach certainly yields a trivial upper or lower bound.  The situation will be somewhat easier and the result are often easily attainable in the best possible framework if we work within a certain space. A typical instance is the Cauchy-Schwartz inequality in the setting of Hilbert space. In light of this we introduce and examine the notion of an outer product and the associated space. We study this space and obtain the following inequalities as a consequence:
\begin{theorem}
Let $\langle V,||.||\rangle$ be a finite dimensional normed-vector space over $\mathbb{R}$  such that $||b||>||a||>1$ for $a,b\in \langle V,||.||\rangle$ and $\mathrm{Spec}(ab^T)\subset \mathbb{R}$ such that $\frac{||b||+||a||}{2}>\bigg|\mathrm{max}\bigg(\mathrm{Spec}(ab^T)\bigg)\bigg|$.  Then the inequality is valid \begin{align}\sum \limits_{\lambda_i\in \mathrm{Spec}(ab^{T})}\mathrm{min}\{\log |t-\lambda_i|\}_{[||a||,||b||]}&\leq \# \mathrm{Spec}(ab^T)\log\bigg(\frac{||b||+||a||}{2}\bigg)\nonumber \\&+\frac{1}{||b||-||a||}\sum \limits_{\lambda_i\in \mathrm{Spec}(ab^T)}\log \bigg(1-\frac{2\lambda_i}{||b||+||a||}\bigg).\nonumber
\end{align}
\end{theorem}
\bigskip

\begin{theorem}
Let $\langle V,||.||\rangle $ be a finite dimensional normed-vector space over $\mathbb{R}$ such that $||b||>||a||>1$ for $a,b\in \langle V,||.||\rangle$ and $\mathrm{Spec}(ab^T)\subset \mathbb{R}$ with  $\frac{||b||+||a||}{2}>\bigg|\mathrm{max}\bigg(\mathrm{Spec}(ab^T)\bigg)\bigg|$. Then the inequality is valid\begin{align}\int \limits_{||a||}^{||b||}\log|\mathrm{det}(ab^{T}-tI)|dt&\leq \# \mathrm{Spec}(ab^T)(||b||-||a||)\log\bigg(\frac{||b||+||a||}{2}\bigg)\nonumber \\&+\sum \limits_{\lambda_i\in \mathrm{Spec}(ab^T)}\log \bigg(1-\frac{2\lambda_i}{||b||+||a||}\bigg).\nonumber
\end{align}
\end{theorem}
\bigskip

\section{\textbf{The outer product}}

\begin{definition}\label{outer}
Let $\bigg\langle V,||.||\bigg \rangle$ be a finite dimensional normed-vector space over $\mathbb{R}$. Then by an outer product on $\langle V,||.||\rangle$, we mean the bivariate map $(;):\langle V,||.||\rangle \times \langle V,||.||\rangle \longrightarrow \mathbb{R}$ such that for any $a,b\in \langle V,||.||\rangle$, then \begin{align}(a;b)=\sum \limits_{\lambda_i\in \mathrm{Spec}(ab^{T})}\int \limits_{||a||}^{||b||}\log |t-\lambda_i|dt,\nonumber
\end{align}where $\mathrm{Spec}(ab^{T})$ denotes the spectrum of the matrix $ab^{T}$. A norm-vector space equipped with an outer product is called an outer product space. In particular, we denote this space by the triple $\bigg \langle V,||.||,(;)\bigg\rangle$.
\end{definition}

\begin{remark}
Technically speaking, the notion of taking the outer product on any two vectors is basically the process of averaging the area of the $\log$ function on each disc with center the elements of the spectrum of the matrix induced by the vectors. One could also think of the outer product space as a much bigger space compared to a norm-vector space; that is, an outer product space is a vector space equipped with the norm and an outer product.
\end{remark}

\section{\textbf{Properties of the outer product}}
The outer product is usually fairly intractable. In this section we examine some elementary properties. 
\begin{proposition}
The following remain valid
\begin{enumerate}
\item [(i)] $(a;b)+(b;a)=0$.
\bigskip

\item [(ii)] $(a;a)=0$.
\bigskip

\item [(iii)] $(b;b)=(a;b)+(b;a)=(a;a)$.
\end{enumerate}
\end{proposition}

\begin{proof}
\item [(i)] By Definition, It follows that \begin{align}(a;b)&=\sum \limits_{\lambda_i\in \mathrm{Spec}(ab^{T})}\int \limits_{||a||}^{||b||}\log |t-\lambda_i|dt,\nonumber \\&=-\sum \limits_{\lambda_i\in \mathrm{Spec}(ba^{T})}\int \limits_{||b||}^{||a||}\log |t-\lambda_i|dt\nonumber \\&=-(b;a)\nonumber
\end{align}since $\mathrm{Spec}(ab^T)=\mathrm{Spec}(ba^T)$, and the result follows immediately.
\bigskip

The result is obvious for $(ii)$ and $(iii)$.
\end{proof}

\bigskip

\section{\textbf{Applications of the outer product}}
In this section we explore an application of the notion of an outer product. We leverage this to obtain a non-trivial upper bound for the average weighted minimum distance of vectors close to the spectrum induced by these vectors. We find the following inequality useful:

\begin{lemma}[Jensen]\label{Jensen}
Let $\mathcal{M}$ be a sigma-algebra of the set $\Omega$ equipped with the positive measure $\mu$ such that $\mu(\Omega)=1$. Let $f$ be a real-valued function in $L^1(\mu)$ with $a<f(x)<b$ for all $x\in \Omega$. If $\varphi$ is concave on $(a,b)$, then we have \begin{align}\int \limits_{\Omega}(\varphi \circ f)d\mu \leq \varphi \bigg(\int \limits_{\Omega}fd\mu \bigg).\nonumber
\end{align}
\end{lemma}
\bigskip

The above Lemma is standard and can be found in several analysis texts (See \cite{rudin2006real}). It has several analogues but this version is appropriate for our needs. The only compromise we make is that we do not use the exact form but a variant. Nevertheless parallels could be drawn after a certain normalization. 

\begin{remark}
The applications we will obtain in the sequel will rely mostly on the assumption that the spectrum  $\mathrm{Spec}(ab^{T}) \subset \mathbb{R}$. This gives the leeway and the freedom to exploit the ordering property of the set of real numbers $\mathbb{R}$, which the complex space $\mathbb{C}$ lacks.
\end{remark}

\begin{proposition}\label{key}
Let $\langle V,||.||\rangle$ be a finite dimensional normed-vector space  over $\mathbb{R}$ and  $a,b\in \bigg \langle V, ||.||,(;)\bigg \rangle$ such that $||b||>||a||>1$ for $a,b\in \langle V,||.||\rangle$ and $\mathrm{Spec}(ab^T)\subset \mathbb{R}$ with $\frac{||b||+||a||}{2}>\bigg|\mathrm{max}\bigg(\mathrm{Spec}(ab^T)\bigg)\bigg|$. Then the inequality remains valid\begin{align}(a;b)&\leq \# \mathrm{Spec}(ab^T)(||b||-||a||)\log\bigg(\frac{||b||+||a||}{2}\bigg)\nonumber \\&+\sum \limits_{\lambda_i\in \mathrm{Spec}(ab^T)}\log \bigg(1-\frac{2\lambda_i}{||b||+||a||}\bigg).\nonumber
\end{align}
\end{proposition}

\begin{proof}
By applying definition \ref{outer} and  applying Lemma \ref{Jensen}, we obtain \begin{align}(a;b)&=\sum \limits_{\lambda_i\in \mathrm{Spec}(ab^{T})}\int \limits_{||a||}^{||b||}\log |t-\lambda_i|dt\nonumber \\&\leq (||b||-||a||)\sum \limits_{\lambda_i\in \mathrm{Spec}(ab^T)}\log \bigg(\int \limits_{||a||}^{||b||}\frac{|t-\lambda_i|}{||b||-||a||}dt\bigg)\nonumber \\&=(||b||-||a||)\sum \limits_{\lambda_i\in \mathrm{Spec}(ab^T)}\log \bigg(\int \limits_{||a||}^{||b||}|t-\lambda_i|dt\bigg)-\# \mathrm{Spec}(ab^T)(||b||-||a||)\log(||b||-||a||)\nonumber 
\end{align}The result follows immediately.
\end{proof}

\begin{remark}
The above inequality relates the outer product to an important matrix statistic. Indeed knowing these enable us to get at least some control on the outer product. Next we examine some immediate consequences of this inequality. It gives the average minimum distance between the element in the spectrum of the matrix induced by the vectors and the vectors.
\end{remark}
\bigskip

\begin{theorem}
Let $\langle V,||.||\rangle$ be a finite dimensional normed-vector space over $\mathbb{R}$ such that $||b||>||a||>1$ for $a,b\in \langle V,||.||\rangle$ and $\mathrm{Spec}(ab^T)\subset \mathbb{R}$ such that $\frac{||b||+||a||}{2}>\bigg|\mathrm{max}\bigg(\mathrm{Spec}(ab^T)\bigg)\bigg|$.  Then the inequality is valid \begin{align}\sum \limits_{\lambda_i\in \mathrm{Spec}(ab^{T})}\mathrm{min}\{\log |t-\lambda_i|\}_{[||a||,||b||]}&\leq \# \mathrm{Spec}(ab^T)\log\bigg(\frac{||b||+||a||}{2}\bigg)\nonumber \\&+\frac{1}{||b||-||a||}\sum \limits_{\lambda_i\in \mathrm{Spec}(ab^T)}\log \bigg(1-\frac{2\lambda_i}{||b||+||a||}\bigg).\nonumber
\end{align}
\end{theorem}

\begin{proof}
By the properties of the outer product, we obtain the lower bound \begin{align}(a;b)&=\sum \limits_{\lambda_i\in \mathrm{Spec}(ab^{T})}\int \limits_{||a||}^{||b||}\log |t-\lambda_i|dt\nonumber \\&\geq \sum \limits_{\lambda_i\in \mathrm{Spec}(ab^T)}(||b||-||a||)\mathrm{min}\{\log |t-\lambda_i|\}_{[||a||,||b||]}\nonumber \\&=||b||-||a||\sum \limits_{\lambda_i\in \mathrm{Spec}(ab^T)}\mathrm{min}\{\log |t-\lambda_i|\}_{[||a||,||b||]}.\nonumber 
\end{align}The result follows by combining this lower bound with the upper bound in Proposition \ref{key}.
\end{proof}
\bigskip

\begin{remark}
Next we establish an inequality for a certain integral of the determinant of matrix induced by vectors in a space. This is a consequence of the upper bound in Proposition \ref{key} and leveraging definition \ref{outer}.
\end{remark}
\bigskip

\begin{theorem}
Let $\langle V,||.||\rangle $ be a finite dimensional normed-vector space over $\mathbb{R}$ such that $||b||>||a||>1$ and $\mathrm{Spec}(ab^T)\subset \mathbb{R}$ with  $\frac{||b||+||a||}{2}>\bigg|\mathrm{max}\bigg(\mathrm{Spec}(ab^T)\bigg)\bigg|$. Then the inequality is valid\begin{align}\int \limits_{||a||}^{||b||}\log|\mathrm{det}(ab^{T}-tI)|dt&\leq \# \mathrm{Spec}(ab^T)(||b||-||a||)\log\bigg(\frac{||b||+||a||}{2}\bigg)\nonumber \\&+\sum \limits_{\lambda_i\in \mathrm{Spec}(ab^T)}\log \bigg(1-\frac{2\lambda_i}{||b||+||a||}\bigg).\nonumber
\end{align}
\end{theorem}

\begin{proof}
This is a simple consequence of Proposition \ref{key}.
\end{proof}

\footnote{
\par
.}%
.




\bibliographystyle{amsplain}

\begin{thebibliography}{10}

\bibitem {rudin2006real}Rudin, W. \textit{Real and complex analysis}, Tata McGraw-hill education, 2006.










\end{thebibliography}

\end{document}